\documentclass{amsart}
\usepackage{amsrefs}
\usepackage[all]{xy}
\usepackage[mathscr]{eucal}
\usepackage{graphicx,amssymb}
\usepackage{pinlabel}
\usepackage{placeins}
\CompileMatrices

\newtheorem{theorem}{Theorem}

\newtheorem{corollary}[theorem]{Corollary}

\theoremstyle{definition}
\newtheorem{definition}[theorem]{Definition}
\newtheorem{example}[theorem]{Example}

\theoremstyle{remark}

\numberwithin{equation}{section}

\begin{document}

\newcommand{\Diff}{\operatorname{Diff}}
\newcommand{\Homeo}{\operatorname{Homeo}}
\newcommand{\Hom}{\operatorname{Hom}}
\newcommand{\Exp}{\operatorname{Exp}}
\newcommand{\Orb}{\operatorname{\textup{Orb}}}
\newcommand{\Emb}{\operatorname{Emb}}

\newcommand{\COrb}{\operatorname{\star\textup{Orb}}}
\newcommand{\CROrb}{\operatorname{\scriptscriptstyle{\blacklozenge}\scriptstyle\textup{Orb}}}
\newcommand{\ssslozenge}{\scriptscriptstyle{\blacklozenge}}
\newcommand{\ssstriangledown}{\scriptscriptstyle{\blacktriangledown}}

\newcommand{\Stwo}{\mbox{$\displaystyle S^2$}}
\newcommand{\Sn}{\mbox{$\displaystyle S^n$}}
\newcommand{\supp}{\operatorname{supp}}
\newcommand{\intr}{\operatorname{int}}
\newcommand{\kernel}{\operatorname{ker}} 

\newcommand{\RR}{\mathbb{R}} 
\newcommand{\ZZ}{\mathbb{Z}} 

\newcommand{\orbify}[1]{\ensuremath{\mathcal{#1}}}
\newcommand{\starfunc}[1]{\ensuremath{{}_\star{#1}}}
\newcommand{\lozengefunc}[1]{\ensuremath{{}_{\scriptscriptstyle{\blacklozenge}}{#1}}}
\newcommand{\redfunc}[1]{\ensuremath{{}_\bullet{#1}}}

\newcommand{\OrbDiff}{\ensuremath{\Diff_{\Orb}}}
\newcommand{\RedOrbDiff}{\ensuremath{\Diff_{\textup{red}}}}
\newcommand{\COrbDiff}{\ensuremath{\Diff_{\COrb}}}
\newcommand{\CROrbDiff}{\ensuremath{\Diff_{\CROrb}}}
\newcommand{\OrbMaps}{\ensuremath{C_{\Orb}}}
\newcommand{\RedOrbMaps}{\ensuremath{C_{\textup{red}}}}
\newcommand{\COrbMaps}{\ensuremath{C_{\COrb}}}
\newcommand{\CROrbMaps}{\ensuremath{C_{\CROrb}}}
\newcommand{\Frechet}{Fr\'{e}chet\ }
\newcommand{\Frechetnospace}{Fr\'{e}chet}
\newcommand*{\longhookrightarrow}{\ensuremath{\lhook\joinrel\relbar\joinrel\rightarrow}}

\title[On the notions of Suborbifold and Orbifold Embedding]{On the notions of Suborbifold and Orbifold Embedding}

\author{Joseph E. Borzellino}

\address{Department of Mathematics, California Polytechnic State
  University, 1 Grand Avenue, San Luis Obispo, California 93407}
\email{jborzell@calpoly.edu}

\author{Victor Brunsden} \address{Department of Mathematics and
  Statistics, Penn State Altoona, 3000 Ivyside Park, Altoona,
  Pennsylvania 16601} \email{vwb2@psu.edu}


\date{\today} \commby{Editor} 
\begin{abstract}
  The purpose of this article is to investigate the relationship between suborbifolds and orbifold embeddings. In particular, we give natural definitions of the notion of suborbifold and orbifold embedding and provide many examples. Surprisingly, we show that there are (topologically embedded) smooth suborbifolds which do not arise as the image of a smooth orbifold embedding. We are also able to characterize those suborbifolds which can arise as the images of orbifold embeddings. As an application, we show that a length-minimizing curve (a geodesic segment) in a Riemannian orbifold can always be realized as the image of an orbifold embedding.
\end{abstract}

\maketitle

\section{Introduction}\label{IntroSection}

The purpose of this article is to investigate some of the difficulties and subtleties associated with the study of the differential topology of smooth orbifolds. It will be no surprise to anyone who has taken more than a cursory look at orbifolds, that the goal of extending the most basic notions from the differential topology of manifolds to orbifolds has not been achieved in a universally accepted manner in the nearly 60 years since Satake \cites{MR0079769,MR0095520} introduced \emph{V-manifolds} (now, \emph{orbifolds} as popularized by Thurston \cite{Thurston78}). In the six decades since they were introduced, there has been a proliferation of definitions and ad-hoc refinements each used to overcome some inherent difficulty unearthed while attempting an orbifold generalization of a manifold result. These challenges are readily acknowledged by experts and often provide the inspiration for new research on orbifolds. In fact, it has been humorously mentioned that there exists today a partial ordering for the plethora of definitions related to orbifolds, and that one can only imagine what an application of Zorn's lemma might yield! The aim here is much less ambitious. Our goal is to expose and investigate in detail the subtle notion of suborbifold and its relation to the natural idea of an orbifold embedding. Some of the particular difficulties involving the notion of suborbifolds and orbifold embeddings have already been noted in the orbifold literature \cites{MR1926425,MR2104605,MR2359514,MR2523149}, and more recently in \cites{MR2778793,MR2962023,MR2973378,MR3126596}. For manifolds, it is a fundamental result of differential topology that submanifolds are precisely the images of embeddings \cite{MR0448362}*{Theorem~3.1}. In fact, many authors use this characterization as the definition of submanifold. Our main result identifies necessary and sufficient conditions which characterize precisely when a suborbifold can be realized as the image of an orbifold embedding. Unlike the case for manifolds, we also show that suborbifolds exist which are not the images of orbifold embeddings.  

\begin{theorem}\label{MainTheorem} Let $\orbify{P}$ be a smooth suborbifold of a smooth orbifold $\orbify{O}$.
\begin{enumerate}
\item\label{MainTheoremPart1} Then, there exists an orbifold $\orbify{P}'$ and a topological embedding of underlying spaces $\iota:X_{\orbify{P}'}\to X_{\orbify{O}}$ so that $\iota(X_{\orbify{P}'})=X_{\orbify{P}}$ if and only if $\orbify{P}$ is saturated.
\item\label{MainTheoremPart2} Moreover, there exists a complete orbifold embedding $\starfunc\iota=(\iota,\{\tilde\iota_x\},\{\Theta_{\iota,x}\}):\orbify{P}'\to\orbify{O}$ covering $\iota$ if and only if $\orbify{P}$ is both saturated and split.
\end{enumerate}
\end{theorem}

The definitions of what it means for a suborbifold to be \emph{saturated} or \emph{split} appear in section~\ref{SuborbifoldSection}. The definition of complete orbifold map appears in 
section~\ref{CompleteOrbiMapSubSection}.

As an application of Theorem~\ref{MainTheorem} to length minimizing geodesics in Riemannian orbifolds, we have the following corollary which follows from the characterization of length minimizing geodesic segments found in \cites{MR2687544,MR1218706}.

\begin{corollary}\label{GeodesicsAreEmbedded} Let $\orbify{O}$ be a Riemannian orbifold and let $X\subset X_{\orbify{O}}$ be the underlying point set of a length-minimizing curve joining two points 
of $\orbify{O}$. Then, there is a suborbifold $\orbify{P}\subset\orbify{O}$ whose underlying space $X_{\orbify{P}}=X$ is the image of a complete orbifold embedding.
\end{corollary}

\section{Orbifold Background}
Although there are many references for this background material, we will use our previous work \cites{MR2523149,MR3117354} as our standard reference. While much of what we discuss here works equally well for smooth $C^r$ orbifolds, to simplify the exposition, we restrict ourselves to smooth $C^\infty$ orbifolds. Throughout, the term \emph{smooth} means $C^\infty$. This results in no loss of generality \citelist{\cite{MR2523149}*{Proposition~3.11} \cite{MR2879379}}. Note that the classical definition of orbifold given below is modeled on the definition in Thurston \cite{Thurston78} and that these orbifolds are referred to as {\em classical effective orbifolds} in \cite{MR2359514}.

\begin{definition}\label{orbifold}
  An $n$-dimensional  \emph{smooth orbifold} $\orbify{O}$,
  consists of a paracompact, Hausdorff topological space
  $X_\orbify{O}$ called the \emph{underlying space}, with the
  following local structure.  For each $x \in X_\orbify{O}$ and
  neighborhood $U$ of $x$, there is a neighborhood $U_x \subset U$, an
  open set $\tilde U_x$ diffeomorphic to $\RR^n$, a finite group $\Gamma_x$ acting
  smoothly and effectively on $\tilde U_x$ which fixes $0\in\tilde
  U_x$, and a homeomorphism $\phi_x:\tilde U_x/\Gamma_x \to U_x$ with
  $\phi_x(0)=x$.  These actions are subject to the condition that for
  a neighborhood $U_z\subset U_x$ with corresponding $\tilde U_z \cong
  \RR^n$, group $\Gamma_z$ and homeomorphism $\phi_z:\tilde
  U_z/\Gamma_z \to U_z$, there is a smooth embedding $\tilde\psi_{zx}:\tilde
  U_z \to \tilde U_x$ and an injective homomorphism
  $\theta_{zx}:\Gamma_z \to \Gamma_x$ so that $\tilde\psi_{zx}$ is
  equivariant with respect to $\theta_{zx}$ (that is, for
  $\gamma\in\Gamma_z, \tilde\psi_{zx}(\gamma\cdot \tilde
  y)=\theta_{zx}(\gamma)\cdot\tilde\psi_{zx}(\tilde y)$ for all
  $\tilde y\in\tilde U_z$), such that the following diagram commutes:
  \begin{equation*}
    \xymatrix{{\tilde U_z}\ar[rr]^{\tilde\psi_{zx}}\ar[d]&&{\tilde U_x}\ar[d]\\
      {\tilde U_z/\Gamma_z}\ar[rr]^>>>>>>>>>>{\psi_{zx}=\tilde\psi_{zx}/\Gamma_z}\ar[dd]^{\phi_z}&&{\tilde U_x/\theta_{zx}(\Gamma_z)\ar[d]}\\
      &&{\tilde U_x/\Gamma_x}\ar[d]^{\phi_x}\\
      {U_z}\ar[rr]^{\subset}&&{U_x} }
  \end{equation*}
\end{definition}

We will refer to the neighborhood $U_x$ or $(\tilde U_x,\Gamma_x)$ or
$(\tilde U_x,\Gamma_x, \rho_x, \phi_x)$ as an \emph{orbifold chart}, and write $U_x=\tilde U_x/\Gamma_x$. In the 4-tuple  notation, we are making explicit the representation $\rho_x:\Gamma_x\to\text{Diff}^\infty(\tilde U_x)$.
The \emph{isotropy group of $x$} is the group $\Gamma_x$. The
definition of orbifold implies that the germ of the action of $\Gamma_x$ in a
neighborhood of the origin of $\RR^n$ is unique, so that by shrinking $\tilde U_x$ if necessary, $\Gamma_x$ is well-defined up to isomorphism. The
\emph{singular set} of \orbify{O} is the set of points
$x \in \orbify{O}$ with $\Gamma_x \ne \{ e\}$.
More detail can be found in \cite{MR2523149}.

\subsection{Smooth Suborbifolds}\label{SuborbifoldSection}
Originally, in \cite{Thurston78}, the notion of an $m$--suborbifold
$\orbify{P}$ of an $n$--orbifold $\orbify{O}$ required
$\orbify{P}$ to be locally modeled on $\RR^m\subset\RR^n$ modulo
finite groups. That is, the local action on $\RR^m$ is induced by the
local action on $\RR^n$. As interest in the differential topology of orbifolds grew, it was discovered early, for instance in \cite{MR1926425}, that this definition was too restrictive to admit, for example, the diagonal embedding of an orbifold as a suborbifold of the product orbifold. Other authors \cites{MR2104605,MR2359514,MR3126596} overcame this difficulty by defining their suborbifolds explicitly as images of their particular notion of orbifold embedding in analogy with the case of manifolds. In \cite{MR2973378}, we defined a notion of suborbifold which is general enough to include the diagonal embedding as a suborbifold of the product, but is independent of our notion of orbifold embedding which we recall in section~\ref{OrbifoldMapNotions}. Using our definition of suborbifold we can also easily identify those suborbifolds in the original sense of Thurston \cite{Thurston78}. We refer to them as \emph{full suborbifolds}. Recall the definition of suborbifold from \cite{MR2973378}:

\begin{definition}\label{SubOrbifold}
  An (embedded) \emph{suborbifold} \orbify{P} of an orbifold \orbify{O} consists
  of the following:
  \begin{enumerate}
  \item A subspace $X_{\orbify{P}}\subset X_{\orbify{O}}$ equipped
    with the subspace topology
  \item\label{FlexibleDef} For each $x\in X_{\orbify{P}}$ and neighborhood $W$ of $x$ in
    $X_{\orbify{O}}$ there is an orbifold chart $(\tilde U_x,
    \Gamma_x, \rho_x, \phi_x)$ about $x$ in \orbify{O} with
    $U_x\subset W$, a subgroup $\Lambda_x \subset \Gamma_x$ of the
    isotropy group of $x$ in \orbify{O} and a $\rho_x(\Lambda_x)$
    invariant submanifold $\tilde V_x\subset \tilde U_x \cong \RR^n$,
    so that $(\tilde V_x, \Lambda_x/\Omega_x,\rho_x\lvert_{\Lambda_x},\psi_x)$ is
    an orbifold chart for $\orbify{P}$ where $\Omega_x=\left\{\gamma\in\Lambda_x\mid \rho_x(\gamma)\lvert_{\tilde{V}_x}=\text{Id}\right\}$. (In particular, the \emph{intrinsic} isotropy subgroup at $x\in\orbify{P}$ is  $\Lambda_x/\Omega_x$), and
  \item
      $V_x = \psi_x(\tilde V_x/\rho_x(\Lambda_x))
      =U_x\cap X_{\orbify{P}}$
    is an orbifold chart for $x$ in $\orbify{P}$.
  \end{enumerate}
\end{definition}

Implicit in this definition is the requirement that the invariant submanifolds $\tilde{V}_x$ be smooth, and that the collection of charts $\{(\tilde V_x, \Lambda_x/\Omega_x,\rho_x\lvert_{\Lambda_x},\psi_x)\}$ satisfy the compatibility conditions of definition~\ref{orbifold}, thus giving $\orbify{P}$ the structure of a smooth orbifold. Condition (\ref{FlexibleDef}) of this definition is not very restrictive as we shall see later in this section. Thurston's notion of suborbifold \cite{Thurston78} is equivalent to adding the condition that $\Lambda_x = \Gamma_x$ at all $x$ in the underlying topological space of $\orbify{P}$, and so we make the following definition:

\begin{definition}\label{FullSuborbifoldDef} $\orbify{P}\subset\orbify{O}$ is a \emph{full suborbifold} of $\orbify{O}$ if $\orbify{P}$ is a suborbifold with
$\Lambda_x=\Gamma_x$ for all $x\in\orbify{P}$.
\end{definition}

When necessary for clarity, we will use the notation $\Gamma_{x,\orbify{O}}$ to denote the intrinsic isotropy of a point $x$ in an orbifold $\orbify{O}$, and use the subscript $\orbify{O}$ as well on needed subgroups of $\Gamma_{x,\orbify{O}}$. Observe that in the case of a suborbifold $\orbify{P}\subset\orbify{O}$ we always have the following exact sequence of groups

$$1\longrightarrow\Omega_{x,\orbify{O}}\longrightarrow\Lambda_{x,\orbify{O}}\subset\Gamma_{x,\orbify{O}}\longrightarrow\Gamma_{x,\orbify{P}}\longrightarrow 1$$
where $\Gamma_{x,\orbify{P}}$ denotes the intrinsic isotropy group of $\orbify{P}$ at $x$.

In characterizing those suborbifolds that are images of orbifold embeddings, we need the following two definitions.

\begin{definition}\label{SplitSuborbifoldDef} We say that $\orbify{P}\subset\orbify{O}$ is a \emph{split} suborbifold of $\orbify{O}$ if the exact sequence above is (right) split for all $x\in\orbify{P}$. That is, there is a group homomorphism $\sigma:\Gamma_{x,\orbify{P}}\to\Lambda_{x,\orbify{O}}$ such that the composition $q\circ\sigma=\text{Id}$, where $q:\Lambda_{x,\orbify{O}}\to\Gamma_{x,\orbify{P}}$ is the quotient homomorphism:

\begin{equation*}
\xymatrix{1\ar[r] & \Omega_{x,\orbify{O}}\ar[r] & \Lambda_{x,\orbify{O}}\ar[r]^-{q} & \Gamma_{x,\orbify{P}}\ar@/^.75pc/[l]^\sigma\ar[r] & 1.}
\end{equation*}
\end{definition}

Note that if $\orbify{P}\subset\orbify{O}$ is split, we have $\Lambda_{x,\orbify{O}}\cong\Omega_{x,\orbify{O}}\rtimes\Gamma_{x,\orbify{P}}$, a semidirect product, and in the case that the groups are abelian $\Lambda_{x,\orbify{O}}\cong\Omega_{x,\orbify{O}}\times\Gamma_{x,\orbify{P}}$, the direct product. Of course, if $\Omega_{x,\orbify{O}}$ or $\Gamma_{x,\orbify{P}}$ is trivial, then $\orbify{P}$ is split as well.

\begin{definition}\label{SaturatedSuborbifoldDef} We say that $\orbify{P}\subset\orbify{O}$ is a \emph{saturated} suborbifold of $\orbify{O}$ if for each $x\in\orbify{P}$ and $\tilde y\in\tilde V_x$, we have that
$\left(\Gamma_{x,\orbify{O}}\cdot\tilde y\right)\cap\tilde V_x=\Lambda_{x,\orbify{O}}\cdot\tilde y$.
\end{definition}

The saturation condition can be thought of as a kind of orbit maximality condition on the group $\Lambda_{x,\orbify{O}}\subset\Gamma_{x,\orbify{O}}$ relative to the invariant submanifold $\tilde V_x$. Observe that, by definition, every full suborbifold is automatically saturated.

\begin{example}\label{FullSuborbifoldExample} Let $\orbify{Q}=\RR/\ZZ_2$ be the smooth orbifold (without boundary)  where $\ZZ_2$ acts on $\RR$ via $\gamma\cdot x=-x$. The underlying topological space $X_{\orbify{Q}}$ of $\orbify{Q}$ is $[0, \infty)$ and the isotropy subgroups are $\{e\}$ for $x\in (0, \infty)$ and $\ZZ_2$ for $x = 0$. Let $\orbify{O}=\orbify{Q}\times\orbify{Q}$ be the smooth product orbifold (without boundary). See \cite{MR2523149}*{Definition~2.12}. The underlying space for $\orbify{O}$ can be identified with the closed first quadrant and the singular points of $\orbify{O}$ lie in one of three connected singular strata: the positive $x$ axis, the positive $y$ axis (corresponding to those points with $\ZZ_2$ isotropy), and the origin which has $\ZZ_2\times\ZZ_2$ isotropy. Then $\orbify{P}=\{0\}\times\orbify{Q}$ is a full (and thus, saturated) suborbifold of $\orbify{O}$. To see this, note that $\Gamma_{(0,0),\orbify{P}}\cong\mathbb{Z}_2$, 
$\Gamma_{(0,0),\orbify{O}}\cong\mathbb{Z}_2\times\mathbb{Z}_2$, and that $\Omega_{(0,0),\orbify{O}}=\{\gamma\in\Gamma_{(0,0),\orbify{O}}:\gamma\mid_{\{0\}\times\mathbb{R}}=\text{Id}\}\cong\mathbb{Z}_2$. Thus,
$\Gamma_{(0,0),\orbify{P}}\cong\Gamma_{(0,0),\orbify{O}}/\Omega_{(0,0),\orbify{O}}$. Similarly, 
$\orbify{P}=\orbify{Q}\times\{0\}$ is a full suborbifold. Each of these suborbifolds is split as well. See figure~\ref{FullSuborbifoldExampleIll}.
\end{example}

\begin{figure}[h]
   \labellist
   \small
   \pinlabel $\orbify{Q}$ [b] at 45 20
   \pinlabel $\orbify{O}$ [b] at 140 25
   \pinlabel $\orbify{P}$  at 98 45
   \pinlabel $\ZZ_2$  at 25 7
   \pinlabel $\ZZ_2\times\ZZ_2$  at 110 -5
   \endlabellist
   \includegraphics[scale=.6]{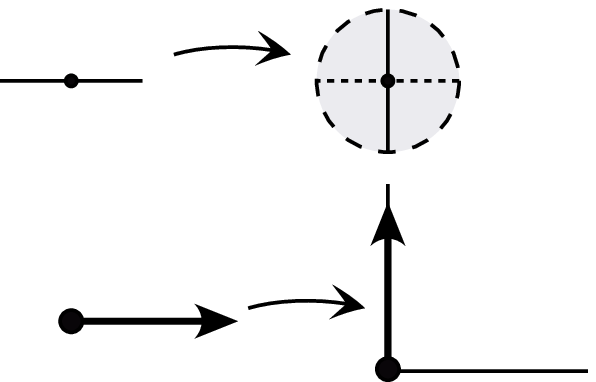}\hspace{1in}
   \labellist
   \small
   \pinlabel $\orbify{Q}$ [b] at 45 20
   \pinlabel $\orbify{O}$ [b] at 160 35
   \pinlabel $\orbify{P}$  at 120 43
   \pinlabel $\ZZ_2$  at 25 6
   \pinlabel $\ZZ_2$  at 137 -7
   \endlabellist
   \includegraphics[scale=.6]{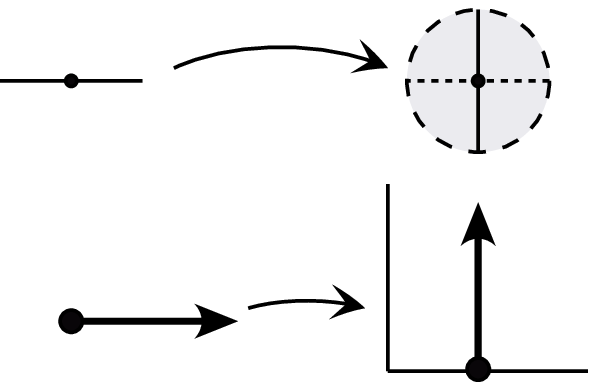}
   \caption{\label{FullSuborbifoldExampleIll}Examples~\ref{FullSuborbifoldExample} and \ref{FullSuborbifoldExample2}}
\end{figure}

\begin{example}\label{FullSuborbifoldExample2} Let $\orbify{Q}$ and $\orbify{O}$ be as in example~\ref{FullSuborbifoldExample}. Then $\orbify{P}=\{1\}\times\orbify{Q}$ is a full (thus, saturated), split suborbifold of $\orbify{O}$. In this case, note that $\Gamma_{(1,0),\orbify{P}}\cong\mathbb{Z}_2$, 
$\Gamma_{(1,0),\orbify{O}}\cong\mathbb{Z}_2$, and that $\Omega_{(1,0),\orbify{O}}=\{\gamma\in\Gamma_{(1,0),\orbify{O}}:\gamma\mid_{\{1\}\times\mathbb{R}}=\text{Id}\}=\{e\}$. Thus,
$\Gamma_{(1,0),\orbify{P}}\cong\Gamma_{(1,0),\orbify{O}}/\Omega_{(1,0),\orbify{O}}$.
See figure~\ref{FullSuborbifoldExampleIll}.
\end{example}

\begin{example}\label{DiagonalSuborbifoldExample} (See \cite{MR2523149}*{Example~2.15}) Let $\orbify{Q}$ and $\orbify{O}$ be as in example~\ref{FullSuborbifoldExample}. Then the diagonal $\orbify{P}=\textup{diag}(\orbify{Q})=\{(x,x)\mid x\in\orbify{Q}\}\subset\orbify{O}$ is a suborbifold. Here, $\Gamma_{(0,0),\orbify{P}}\cong\mathbb{Z}_2$, $\Gamma_{(0,0),\orbify{O}}\cong\mathbb{Z}_2\times\mathbb{Z}_2$, and $\Omega_{(0,0),\orbify{O}}=\{\gamma\in\Gamma_{(0,0),\orbify{O}}:\gamma\mid_{\textup{diag}(\mathbb{R})\subset\RR^2}=\text{Id}\}=\{e\}$. Note that
$\Gamma_{(0,0),\orbify{P}}\ncong\Gamma_{(0,0),\orbify{O}}/\Omega_{(0,0),\orbify{O}}$. Thus, $\orbify{P}$ is not a full suborbifold. However, $\orbify{P}$ is split and saturated since $\Gamma_{(0,0),\orbify{P}}\cong\Lambda_{(0,0),\orbify{O}}/\Omega_{(0,0),\orbify{O}}$, where $\mathbb{Z}_2\cong\Lambda_{(0,0),\orbify{O}}\subset\Gamma_{(0,0),\orbify{O}}$ is the diagonal embedding of
$\mathbb{Z}_2\hookrightarrow\mathbb{Z}_2\times\mathbb{Z}_2$ given by $\gamma\mapsto (\gamma,\gamma)$.
See figure~\ref{DiagonalSuborbifoldExampleIll}.
\end{example}

\begin{figure}[h]
   \labellist
   \small
   \pinlabel $\orbify{Q}$ [b] at 45 20
   \pinlabel $\orbify{O}$ [b] at 164 35
   \pinlabel $\orbify{P}$  at 124 31
   \pinlabel $\ZZ_2$  at 25 7
   \pinlabel $\ZZ_2\subset\ZZ_2\times\ZZ_2$ [l] at 100 -6
   \endlabellist
   \includegraphics[scale=.6]{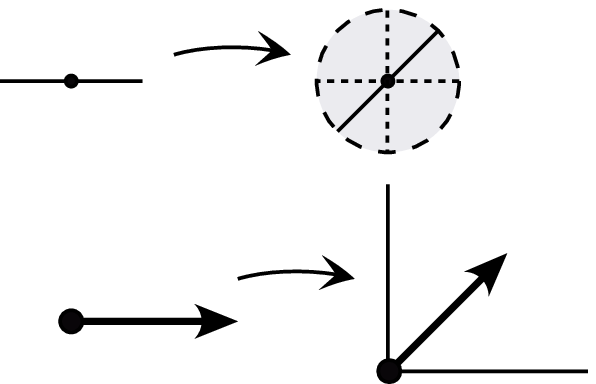}\hspace{1in}
   \labellist
   \small
   \pinlabel $\{e\}$ at 21 95
   \pinlabel $\orbify{S}$ at 43 145
   \pinlabel $\orbify{S}$ at 177 32
   \pinlabel $\orbify{O}$  at 120 45
   \pinlabel $\{e\}\subset\ZZ_2$ at 155 -6
   \endlabellist
   \includegraphics[scale=.6]{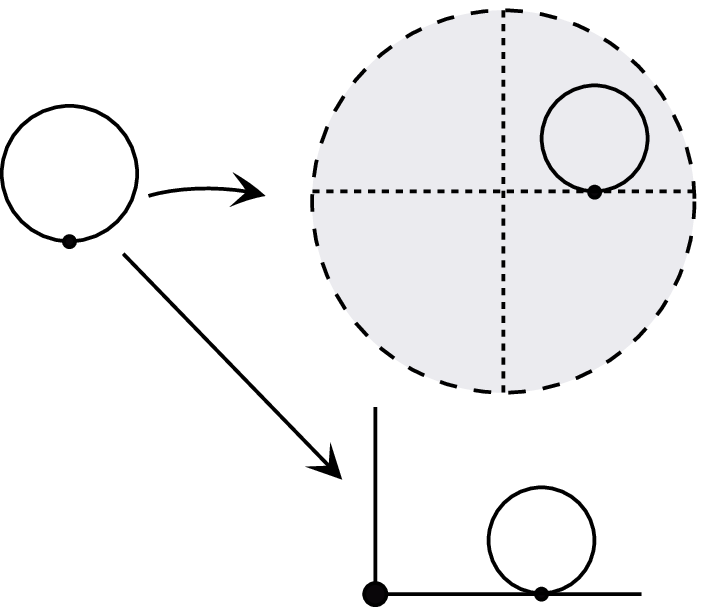}
   \caption{\label{DiagonalSuborbifoldExampleIll}Examples~\ref{DiagonalSuborbifoldExample} and \ref{SuborbifoldExample}}
\end{figure}

\begin{example}\label{SuborbifoldExample}
Let \orbify{O} be as in example~\ref{FullSuborbifoldExample}. Consider the circle $\orbify{S}\subset\orbify{O}$ of radius 1 centered at $(2,1)$. Then $\orbify{S}$ is a suborbifold of $\orbify{O}$ that is not a full suborbifold. To see this, just note that at the point $x=(2,0)\in\orbify{O}$ any lift of $\orbify{S}$ to $\tilde U_x\cong\RR^2$ in a neighborhood of $x$, cannot be an invariant submanifold unless we choose $\Lambda_{x,\orbify{O}}=\{e\}$. In this case, we see that the intrinsic isotropy group of $\orbify{S}$ at $x$ is trivial which it must be since $\orbify{S}$ is actually a compact 1-dimensional \emph{manifold}. That is, a compact 1-dimensional orbifold with trivial orbifold structure. It is easy to see that $\orbify{S}$ is saturated and split as well.
See figure~\ref{DiagonalSuborbifoldExampleIll}.
\end{example}

Each of the previous examples will be seen to be the image of an orbifold embedding in section~\ref{OrbEmbeddingsSection}. However, the following three examples of suborbifolds will be shown not to be the image of an orbifold embedding.

\begin{example}\label{NonEmbeddableSuborbifold} Let $\orbify{O}=\mathbb{C}^2/\mathbb{Z}_4$, where $\mathbb{Z}_4$ acts on $\mathbb{C}^2$ via the matrix group
$$\mathbb{Z}_4\cong\left\{\begin{pmatrix}i^k & 0\\0 & {(-1)}^k\end{pmatrix}\,: k\in\{0,1,2,3\}\right\}.$$
Then $\orbify{P}=(\{0\}\times\mathbb{C})/\mathbb{Z}_4\cong\mathbb{C}/\mathbb{Z}_2$ is a full, (hence, saturated) suborbifold of $\orbify{O}$. In this case, for $x=(0,0)$, we have $\Omega_{x,\orbify{O}}\cong\mathbb{Z}_2$, $\Lambda_{x,\orbify{O}}=\Gamma_{x,\orbify{O}}\cong\mathbb{Z}_4$, and $\Gamma_{x,\orbify{P}}\cong\mathbb{Z}_2$. Hence, the corresponding suborbifold exact sequence for $\orbify{P}$ is $1\rightarrow\mathbb{Z}_2\rightarrow\mathbb{Z}_4\rightarrow\mathbb{Z}_2\rightarrow 1$ which is clearly not split.
\end{example}

The next example illustrates how flexible the seemingly straightforward definition of suborbifold actually is.

\begin{example}\label{SplitOnlySubOrbifolds} Consider the 2-dimensional orbifolds $\orbify{O}_1=\mathbb{C}=\mathbb{C}/\mathbb{Z}_1,\orbify{O}_2=\mathbb{C}/\mathbb{Z}_2$, $\orbify{O}_3=\mathbb{C}/\mathbb{Z}_4$, and $\orbify{O}_4=\mathbb{C}/\mathbb{Z}_8$. Here $\mathbb{Z}_k$ acts on $\mathbb{C}$ via multiplication by $e^{2\pi i/k}$, $z\mapsto e^{2\pi i/k}z$. According to definition~\ref{SubOrbifold}, we have
$\orbify{O}_1\subset\orbify{O}_2\subset\orbify{O}_3\subset\orbify{O}_4$ as suborbifolds. The underlying topological spaces $X_{\orbify{O}_n}$ are all (topologically) homeomorphic, to a standard cone over a circle. It is easily checked that none of these is a saturated suborbifold of one of the others and since $\Omega_{0,\orbify{O}_n}=\{e\}$, $\orbify{O}_m\subset\orbify{O}_n$, $(m<n)$ are all split suborbifolds.
\end{example}

\begin{example}\label{NonFullNonSplitSaturatedSuborbifold} Let $\orbify{O}$ and $\orbify{P}$ be as in example~\ref{NonEmbeddableSuborbifold}. Let 
$\orbify{Q}=\text{diag}(\orbify{P})=\{(x,x)\in\orbify{O}\times\orbify{O}\mid x\in\orbify{P}\}$. Let $\orbify{R}=\orbify{O}\times\orbify{O}$. Let $x=(0,0)$. Then $\Gamma_{x,\orbify{R}}\cong\ZZ_4\times\ZZ_4$. Analogous to example~\ref{DiagonalSuborbifoldExample}, we see that $\Gamma_{x,\orbify{Q}}\cong\ZZ_2$, $\Lambda_{x,\orbify{R}}\cong\ZZ_4$, and $\Omega_{x,\orbify{R}}\cong\ZZ_2$. Thus, $\orbify{Q}$ is not split in $\orbify{R}$ and is not a full suborbifold of $\orbify{R}$. On the other hand, it is not hard to see that $\orbify{Q}$ is saturated in $\orbify{R}$.
\end{example}

Our last example shows that even though the underlying space of a smooth orbifold (without boundary) may be topologically embedded as a subspace of the underlying space of another smooth orbifold, this subspace (with its independent orbifold structure) may not be a suborbifold of the ambient orbifold.

\begin{example}\label{NonSuborbifold} Let $\orbify{O}$ be a so-called $\ZZ_p$-teardrop, and let $\orbify{Q}=\RR/\ZZ_2$ be as in example~\ref{FullSuborbifoldExample}, a smooth one-dimensional orbifold (without boundary). Let $X_\orbify{Q}\subset X_\orbify{O}$ be topologically embedded as a half-interval starting at the point $x$. See figure~\ref{TeardropFigure}. As a non-trivial 1-orbifold, the intrinsic isotropy group for $\orbify{Q}$ at $x$ must be $\ZZ_2$. Thus, the order of the ambient isotropy group, $\Gamma_{x,\orbify{O}}$, must be even. We conclude that $\orbify{Q}$ is not a suborbifold of $\orbify{O}$ when $p$ is odd. Of course, if $p$ is even, then it is possible for $\orbify{Q}$ to be a suborbifold of $\orbify{O}$ which is saturated and split and thus an embedded suborbifold.
\end{example}

\begin{figure}[h]
\labellist
   \small
   \pinlabel $\ZZ_p$ at 53 208
   \pinlabel $\orbify{Q}$  at 51 158
\endlabellist
   \includegraphics[scale=.5]{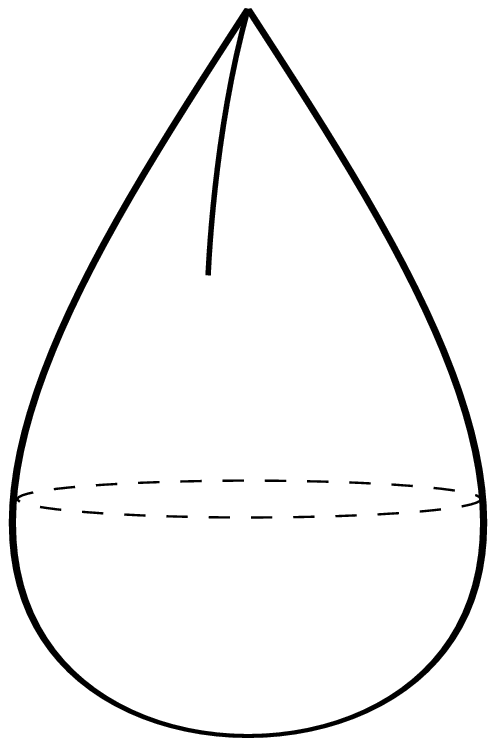}
   \caption{\label{TeardropFigure}Example~\ref{NonSuborbifold}}
\end{figure}

\section{Smooth Mappings Between Orbifolds}\label{OrbifoldMapNotions}

In the literature, there are four related definitions of maps between orbifolds which are based on the classical Satake-Thurston approach to orbifolds via atlases of orbifold charts. In this paper, we use the notion of complete orbifold map. It is distinguished from the other notions of orbifold map in that it keeps track of all defining data. All other notions of orbifold map descend from the complete orbifold maps by forgetting information. In the special case of embeddings, however, the property of being an embedding passes down from the complete orbifold maps to the level of orbifold maps. This observation requires only an understanding on how these two notions of orbifold map are related to one another. We point this out explicitly in our exposition below. We refer the reader to \cite{MR3117354} for the necessary background details and in what follows we use the notation of \cite{MR2523149}*{Section 2}.

The original motivation for defining the notion of complete orbifold map was to make meaningful and well-defined certain geometric constructions involving orbifolds and their maps. The need to be careful in defining an adequate notion of orbifold map was already noted in the work of Moerdijk and Pronk \cite{MR1466622} and Chen and Ruan \cite{MR1950941} and was missing from Satake's original work on $V$-manifolds \cites{MR0079769,MR0095520}. More recently, Pohl \cite{Pohl2013}, developed another notion of orbifold morphism to address some inconsistencies in earlier work using the groupoid approach to orbifolds.

\subsection{Mappings Between Orbifolds}\label{CompleteOrbiMapSubSection}
\begin{definition}\label{CompleteOrbiMap}  
  A $C^\infty$ \emph{complete orbifold map} $\starfunc f=(f,\{\tilde f_x\},\{\Theta_{f,x}\})$ between
  smooth orbifolds $\orbify{P}$ and $\orbify{O}$ consists of the
  following:
  \begin{enumerate}
  \item A continuous map $f:X_{\orbify{P}}\to X_{\orbify{O}}$ of
    the underlying topological spaces.
  \item For each $y\in \orbify{P}$, a group homomorphism
    $\Theta_{f,y}:\Gamma_y\to\Gamma_{f(y)}$.
  \item A smooth $\Theta_{f,y}$-equivariant lift $\tilde f_y:\tilde
    U_y\to\tilde V_{f(y)}$ where
    $(\tilde U_y,\Gamma_y)$ is an
    orbifold chart at $y$ and $(\tilde V_{f(y)},\Gamma_{f(y)})$ is an orbifold chart at $f(y)$.  That
    is, the following diagram commutes:
    \begin{equation*}
      \xymatrix{{\tilde U_y}\ar[rr]^{\tilde f_y}\ar[d]&&{\tilde V_{f(y)}}\ar[d]\\
        {\tilde U_y}/\Gamma_y\ar[rr]^>>>>>>>>>>%
        {{\tilde f_y}/\Theta_{f,y}(\Gamma_y)}\ar[dd]&&{\tilde
          V_{f(y)}}/\Theta_{f,y}(\Gamma_y)\ar[d]\\
        &&{\tilde V_{f(y)}}/\Gamma_{f(y)}\ar[d]\\
        U_y\ar[rr]^{f}&&V_{f(y)}
      }
    \end{equation*}
\begin{sloppypar}
  \item[($\star 4$)] (Equivalence) Two complete orbifold maps $\starfunc f=(f,\{\tilde f_x\},\{\Theta_{f,x}\})$ and
    $\starfunc g=(g,\{\tilde g_x\},\{\Theta_{g,x}\})$ are considered equivalent if for each
    $x\in\orbify{P}$, $\tilde f_x=\tilde g_x$ as germs and $\Theta_{f,x}=\Theta_{g,x}$. That is,
    there exists an orbifold chart $(\tilde U_x,\Gamma_x)$ at $x$ such
    that ${\tilde f_x}\vert_{\tilde U_x}={\tilde g_x}\vert_{\tilde
      U_x}$ and $\Theta_{f,x}=\Theta_{g,x}$. Note that this implies that $f=g$.
\end{sloppypar}
  \end{enumerate}
  The set of smooth complete orbifold maps from $\orbify{P}$ to $\orbify{O}$ will be denoted by
$C^\infty_{\COrb}(\orbify{P}, \orbify{O})$. For $\orbify{P}$ compact (without boundary), $C^\infty_{\COrb}(\orbify{P}, \orbify{O})$ carries the structure of a smooth \Frechet manifold \cite{MR3117354}.
\end{definition}

If we replace ($\star 4$) in definition~\ref{CompleteOrbiMap} by

\begin{enumerate}
  \begin{sloppypar}
  \item[(4)] (Equivalence) Two complete orbifold maps $(f,\{\tilde
    f_x\},\{\Theta_{f,x}\})$ and $(g,\{\tilde g_x\},\{\Theta_{g,x}\})$
    are considered equivalent if for each $x\in\orbify{P}$, $\tilde
    f_x=\tilde g_x$ as germs. That is, there exists an orbifold chart
    $(\tilde U_x,\Gamma_x)$ at $x$ such that ${\tilde
      f_x}\vert_{\tilde U_x}={\tilde g_x}\vert_{\tilde U_x}$ (which as
    before implies $f=g$),
  \end{sloppypar}
\end{enumerate}
where we have dropped the requirement that
$\Theta_{f,x}=\Theta_{g,x}$, we recover the notion of {\em orbifold
  map} $(f,\{\tilde f_x\})$ which appeared in \cite{MR2523149}*{Section
  3}. Thus, the set of orbifold maps $C^\infty_{\Orb}(\orbify{P},
\orbify{O})$ can be regarded as the equivalence classes of complete
orbifold maps under the less restrictive set-theoretic equivalence
($4$). For $\orbify{P}$ compact (without boundary), $C^\infty_{\Orb}(\orbify{P}, \orbify{O})$ carries the structure of a stratified space whose strata are modeled on smooth \Frechet manifolds \cite{MR3117354}.

\subsection{Orbifold Embeddings}\label{OrbEmbeddingsSection}

\begin{definition}\label{CompleteOrbiEmbedding} A complete orbifold map $\starfunc f=(f,\{\tilde f_x\},\{\Theta_{f,x}\})$ between smooth orbifolds $\orbify{P}$ and $\orbify{O}$ is a complete orbifold \emph{embedding} if the map $f:X_\orbify{P}\to X_\orbify{O}$ is a topological embedding of the underlying spaces, each of the homomorphisms $\Theta_{f,y}:\Gamma_y\to\Gamma_{f(y)}$ is injective, and on each chart, the $\Theta_{f,y}$-equivariant local lifts $\tilde f_y:\tilde U_y\to\tilde V_{f(y)}$ are smooth embeddings.
\end{definition}

One should observe that the condition that the equivariant local lifts $\tilde f_x$ are embeddings automatically implies that the corresponding homomorphisms $\Theta_{f,x}$ are injective. For, if there exists $\gamma\in\Gamma_x$ with $\Theta_{f,x}(\gamma)=\{e\}$, then equivariance of $\tilde f_x$ yields 
$\tilde f_x(\gamma\cdot\tilde y)=\Theta_{f,x}(\gamma)\cdot\tilde f_x(\tilde y)=\tilde f_x(\tilde y)$ for all $\tilde y\in\tilde U_x$. Since $\tilde f_x$ is an embedding this implies that $\gamma=\{e\}$, and thus $\Theta_{f,x}$ is injective. Thus, the condition that $\Theta_{f,y}$ be injective is redundant for embeddings. As a consequence, there is a sensible definition of orbifold embedding in the category of orbifold maps as well:

\begin{definition}\label{OrbiEmbedding} An orbifold map $f=(f,\{\tilde f_x\})$ between smooth orbifolds $\orbify{P}$ and $\orbify{O}$ is an orbifold \emph{embedding} if the map $f:X_\orbify{P}\to X_\orbify{O}$ is a topological embedding of the underlying spaces, and on each chart, the $\Theta_{f,y}$-equivariant local lifts $\tilde f_y:\tilde U_y\to\tilde V_{f(y)}$ are smooth embeddings.
\end{definition}

The following example from \cite{MR3117354}*{section 2} is illustrative.

\begin{example}\label{TwoCompleteEmbRepresentSameOrbEmb} Let $\orbify{Q}=\mathbb{R}/\mathbb{Z}_2$ be as in example~\ref{FullSuborbifoldExample}. Consider the inclusion (embedding)
  $f:\orbify{Q}\to\orbify{Q}\times\orbify{Q}\times\orbify{Q}$,
  $y\mapsto (y,0,0)$, where $\tilde f_x(\tilde y)=(\tilde
  y,0,0)$. Note that $\tilde f_0$ is equivariant with respect to both
  $\Theta_{f,0}(\gamma)=(\gamma,e,e)$ and
  $\Theta'_{f,0}(\gamma)=(\gamma,\gamma,\gamma)$. Thus, we have two distinct complete orbifold embeddings
  $\starfunc f=(f,\{\tilde f_x\},\{\Theta_{f,x}\})$ and $\starfunc f'=(f,\{\tilde f_x\},\{\Theta'_{f,x}\})$ which represent the same orbifold embedding $f=(f,\{\tilde f_x\})$. In each case, observe that both $\Theta_{f,x}$ and $\Theta'_{f,x}$ are injective confirming the remarks which followed definition~\ref{CompleteOrbiEmbedding}.\end{example}
  
For open embeddings, that is, in the case where $\dim(\orbify{P})=\dim(\orbify{O})$, it is useful to note that the phenomenon in example~\ref{TwoCompleteEmbRepresentSameOrbEmb} cannot occur \cite{MR3117354}*{section 4}. To see this, note that if two complete orbifold embeddings $\starfunc f=(f,\{\tilde f_x\},\{\Theta_{f,x}\})$ and $\starfunc f'=(f,\{\tilde f_x\},\{\Theta'_{f,x}\})$ represent the same orbifold embedding $f=(f,\{\tilde f_x\})$, then equivariance of $\tilde f_x$ implies that $\tilde f_x(\gamma\cdot\tilde y)=\Theta_{f,x}(\gamma)\cdot\tilde f_x(\tilde y)=\Theta'_{f,x}(\gamma)\cdot\tilde f_x(\tilde y)$ for all $\tilde y\in\tilde U_x$ and $\gamma\in\Gamma_x$. Thus, $[\Theta'_{f,x}(\gamma)^{-1}\Theta_{f,x}(\gamma)]\cdot\tilde f_x(\tilde y)=\tilde f_x(\tilde y)$. Openness of the embedding implies that there exists $\tilde y$ such that $\tilde f_x(\tilde y)$ is not a singular point of $\tilde V_{f(x)}$. This implies that 
$\Theta'_{f,x}(\gamma)^{-1}\Theta_{f,x}(\gamma)=e$ since $\Gamma_{f(x)}$ acts effectively, whence $\Theta_{f,x}=\Theta'_{f,x}$ and $\starfunc f=\starfunc f'$.

\section{Proof of theorem~\ref{MainTheorem} and Corollary~\ref{GeodesicsAreEmbedded}}

\begin{proof}[Proof of part (\ref{MainTheoremPart1})] 
For each $x\in X_{\orbify{P}}$, let $(\tilde U_x,\Gamma_{x,\orbify{O}},\rho_x,\phi_x)$ be an orbifold chart for $\orbify{O}$ about $x$.  Let $\tilde V_x\subset\tilde U_x$, $\Lambda_{x,\orbify{O}}$, $\Omega_{x,\orbify{O}}$, $\psi_x$, and $\Gamma_{x,\orbify{P}}$ be as in the definition of suborbifold. Denote by $\tilde i_x:\tilde V_x\longhookrightarrow\tilde U_x$ the inclusion map.

Let $q_x:\tilde U_x/\Lambda_{x,\orbify{O}}\to\tilde U_x/\Gamma_{x,\orbify{O}}$ be the natural quotient map and define $\phi_x'=\phi_x\circ\ q_x\circ \ {\tilde i_x}/\Lambda_{x,\orbify{O}}:\tilde V_x/\Lambda_{x,\orbify{O}}\to U_x$: 

\begin{equation*}
      \xymatrix{{\tilde V_x}\ar[rr]^{\tilde i_x}\ar[d]&&{\tilde U_x}\ar[d]\\
        {\tilde V_x}/\Lambda_{x,\orbify{O}}\ar[rr]^>>>>>>>>>>%
        {{\tilde i_x}/\Lambda_{x,\orbify{O}}}\ar[ddrr]_
        {\begin{aligned}
        \scriptstyle {\phi_x'= }& \\[-1ex] \scriptstyle{\phi_x\circ\ q_x\circ } & \scriptstyle{\ {\tilde i_x}/\Lambda_{x,\orbify{O}}}\end{aligned}}
        &&{\tilde U_x}/\Lambda_{x,\orbify{O}}\ar[d]^{q_x}\\
        &&{\tilde U_x}/\Gamma_{x,\orbify{O}}\ar[d]^{\phi_x}\\
        && U_x 
      }
    \end{equation*}
For $\tilde y\in\tilde V_x$, let 
$\Lambda_{x,\orbify{O}}(\tilde y)=\{\gamma\in\Gamma_{x,\orbify{O}}:\gamma\cdot\tilde y\in\tilde V_x\}.$
Then, 
$\Lambda_{x,\orbify{O}}\subset\bigcap_{\tilde y\in\tilde V_x}\Lambda_{x,\orbify{O}}(\tilde y).$
By definition, $\orbify{P}$ is saturated if and only if 
$\Lambda_{x,\orbify{O}}\cdot\tilde y=\Lambda_{x,\orbify{O}}(\tilde y)\cdot\tilde y$ for all $\tilde y\in\tilde V_x$.  We claim that $\phi'_x$ is a homeomorphism onto its image if and only if $\orbify{P}$ is saturated. To see this, note that  $\orbify{P}$ is not saturated if and only if there exists $x\in X_{\orbify{P}}$ and $\delta\in\Gamma_{x,\orbify{O}}$ so that for some $\tilde y\in\tilde V_x$, $\delta\cdot\tilde y\in\tilde V_x$, but $\delta\cdot\tilde y\ne\lambda\cdot\tilde y$ for any $\lambda\in\Lambda_{x,\orbify{O}}$. Thus, $\tilde z=\delta\cdot\tilde y$ satisfies $\Lambda_{x,\orbify{O}}\cdot\tilde z\ne\Lambda_{x,\orbify{O}}\cdot\tilde y$. Since 
$\phi'_x(\Lambda_{x,\orbify{O}}\cdot\tilde z)=\phi'_x(\Lambda_{x,\orbify{O}}\cdot\tilde y)$, because $\tilde y$ and $\tilde z$ are in the same orbit under the full group $\Gamma_{x,\orbify{O}}$, we see that $\phi'_x$ is not a homeomorphism. Thus, we have shown that if $\phi'_x$ is a homeomorphism, then $\orbify{P}$ is saturated.
To show that $\orbify{P}$ is saturated implies $\phi'_x$ is a homeomorphism, note that $\phi'_x$ is clearly continuous, and since $\Gamma_{x,\orbify{O}}$ is finite, $\phi_x'$ is open. As shown above, $\orbify{P}$ saturated implies that $\phi'_x$ is injective.

Now, since $\Omega_{x,\orbify{O}}$ fixes $\tilde V_x$ by definition, there is a natural identification $I:\tilde V_x/\Lambda_{x,\orbify{O}}\leftrightarrow\tilde V_x/(\Lambda_{x,\orbify{O}}/\Omega_{x,\orbify{O}})=\tilde V_x/\Gamma_{x,\orbify{P}}$ and we have the diagram:

\begin{equation*}
      \xymatrix{{\tilde V_x}\ar[rr]^{\tilde i_x}\ar[d]&&{\tilde U_x}\ar[d]\\
        {\tilde V_x}/\Lambda_{x,\orbify{O}}\ar[rr]^>>>>>>>>>>%
        {{\tilde i_x}/\Lambda_{x,\orbify{O}}}\ar[ddrr]_{\phi_x'}\ar@{=}[d]_I
        &&{\tilde U_x}/\Lambda_{x,\orbify{O}}\ar[d]^{q_x}\\
        \tilde V_x/\Gamma_{x,\orbify{P}}\ar[d]_{\psi_x} &&{\tilde U_x}/\Gamma_{x,\orbify{O}}\ar[d]^{\phi_x}\\
        V_x\ar[rr]^{\iota}&& U_x 
      }
    \end{equation*}
    
Let $\orbify{P}'$ be the orbifold defined by the local charts $(\tilde V_x,\Gamma_{x,\orbify{P}},\rho'_x,\psi_x)$, where
$\rho'_x$ is the induced action of $\Gamma_{x,\orbify{P}}$ on $\tilde V_x$ by restricting $\rho_x$ to $\Lambda_{x,\orbify{O}}$ and the action to $\tilde V_x$. The required topological embedding $\iota:X_{\orbify{P}'}\to X_{\orbify{O}}$ with $\iota(X_{\orbify{P}'})=X_{\orbify{P}}$ is given in local charts by:
$\iota=\phi_x'\circ I^{-1}\circ\psi_x^{-1}:V_x\to U_x$ and is covered by the inclusion maps $\tilde i_x$. This completes the proof of part (\ref{MainTheoremPart1}).    
\end{proof}

\begin{proof}[Proof of part (\ref{MainTheoremPart2})]
Let $\sigma_{x,\orbify{P}}$ be a splitting of the exact sequence

\begin{equation*}
\xymatrix{1\ar[r] & \Omega_{x,\orbify{O}}\ar[r] & \Lambda_{x,\orbify{O}}\overset{\iota_{x,\orbify{O}}}{\longhookrightarrow}\Gamma_{x,\orbify{O}} \ar[r]^-{q}\ar@{}[]!<-8ex,0ex>="a" & \Lambda_{x,\orbify{O}}/\Omega_{x,\orbify{O}}=\Gamma_{x,\orbify{P}}\ar@{}[]!<12ex,0ex>="b" \ar[r]\ar@/^1.5pc/"b";"a"^{\sigma_{x,\orbify{P}}}& 1.}
\end{equation*}

Let $\Theta_{\iota,x}=\iota_{x,\orbify{O}}\,\circ\,\sigma_{x,\orbify{P}}:\Gamma_{x,\orbify{P}}\to\Gamma_{x,\orbify{O}}$ where  $\iota_{x,\orbify{O}}=\Lambda_{x,\orbify{O}}\longhookrightarrow\Gamma_{x,\orbify{O}}$ is the inclusion map. $\Theta_{\iota,x}$ is clearly an injective homomorphism and note that the existence of $\Theta_{\iota,x}$ is equivalent to the existence of $\sigma_{x,\orbify{P}}$. Let $\tilde i_x:\tilde V_x\longhookrightarrow\tilde U_x$ be the inclusion map. It is easy to see that $\tilde i_x$ is $\Theta_{\iota,x}$-equivariant:  Let $\tilde y\in\tilde V_x$ and let $\gamma\in\Gamma_{x,\orbify{P}}$. Then, $\gamma\cdot\tilde y=\tilde i_x(\gamma\cdot\tilde y)$. On the other hand, $\Theta_{\iota,x}(\gamma)= \gamma\omega$ for some $\omega\in\Omega_{x,\orbify{O}}$ since $q\circ\Theta_{\iota,x}=\text{Id}$. Thus, 
$\Theta_{\iota,x}(\gamma)\cdot\tilde y=\gamma\omega\cdot\tilde y=\gamma\cdot\tilde y$, since $\omega\cdot\tilde y=\tilde y$. Thus, $\starfunc \iota=(\iota,\{\tilde i_x\},\{\Theta_{\iota,x}\}):\orbify{P}'\to \orbify{P}\subset\orbify{O}$ is a complete orbifold embedding that covers the topological embedding $\iota:X_{\orbify{P'}}\to X_{\orbify{P}}\subset X_{\orbify{O}}$. This completes the proof of part (\ref{MainTheoremPart2}).
\end{proof}

\begin{proof}[Proof of Corollary~\ref{GeodesicsAreEmbedded}] $X$ is the point set of a length-minimizing curve in $\orbify{O}$, so there exists a unit-speed parametrization $c:[a,b]\to\orbify{O}$ such that $c([a,b])=X$. We show that $X$ is the underlying space of a suborbifold $\orbify{P}\subset\orbify{O}$ that is saturated and split and thus is the image of an orbifold embedding by theorem~\ref{MainTheorem}. 
Fix $t\in [a,b]$, and let $x=c(t)$. It follows from the characterization of length minimizing geodesic segments in \cite{MR2687544}*{theorem~3, pg. 32} or \cite{MR1218706}*{proposition~15} that $c$ is contained in the closure of single connected open stratum of $\orbify{O}$. That is, $c|_{(a,b)}$ lies in a subspace of $\orbify{O}$ with constant isotropy. This implies that $X$ has the structure of a suborbifold $\orbify{P}\subset\orbify{O}$. These results also imply that $\Gamma_{x,\orbify{P}}=\{e\}$, and thus $\orbify{P}$ has a trivial orbifold structure and so $\orbify{P}$ is split. Let $\tilde c_x$ be a lift of $c$ to $\tilde U_x$. Since $\Gamma_{x,\orbify{P}}=\{e\}$, we have $\Lambda_{x,\orbify{O}}\cdot\tilde y=\tilde y$ for all $\tilde y\in\tilde c_x$. If $\orbify{P}$ were not saturated at $x$, then there exists $s<s'\in [a,b]$ and $\gamma\in\Gamma_{x,\orbify{O}}$ with $\gamma\cdot \tilde c_x(s)=\tilde c_x(s')$. This implies that $c$ contains a loop in $\orbify{O}$. This contradicts the property that a length-minimizing curve $c$ must minimize length between any of its points. This contradiction implies $\orbify{P}$ is saturated and thus, by theorem~\ref{MainTheorem},  $\orbify{P}$ is the image of a complete orbifold embedding. This completes the proof of corollary~\ref{GeodesicsAreEmbedded}.
\end{proof}

\section{Table of Examples and Suborbifold Properties}

Here is a summary table of properties of suborbifolds possessed by the examples presented in this article which show that all possible combinations of properties that are not implied by others can occur:
\begin{table}[h!]
\centering
\begin{tabular}{r|c|c|c|c|c|}
Example & Suborbifold & Full & Saturated & Split & Image of Orbifold Embedding\\ \hline
\ref{FullSuborbifoldExample} & Yes & Yes & Yes & Yes & Yes\\ \cline{2-6}
\ref{FullSuborbifoldExample2} & Yes & Yes & Yes & Yes & Yes\\ \cline{2-6}
\ref{DiagonalSuborbifoldExample} & Yes & No & Yes & Yes & Yes\\ \cline{2-6}
\ref{SuborbifoldExample} & Yes & No & Yes & Yes & Yes\\ \cline{2-6}
\ref{NonEmbeddableSuborbifold} & Yes & Yes & Yes & No & No\\ \cline{2-6}
\ref{SplitOnlySubOrbifolds} & Yes & No & No & Yes & No\\ \cline{2-6}
\ref{NonFullNonSplitSaturatedSuborbifold} & Yes & No & Yes & No & No\\ \cline{2-6}
\ref{NonSuborbifold} & No & -- & -- & -- & No \\ \cline{2-6}
\end{tabular}
\label{OrbifoldTable}
\end{table}

\FloatBarrier

\section{Acknowledgment} We are grateful to the referee and wish to thank him for his useful comments and valuable corrections that lead to improvement of the final manuscript.


\bib{MR2359514}{book}{
  author={Adem, Alejandro},
  author={Leida, Johann},
  author={Ruan, Yongbin},
  title={Orbifolds and stringy topology},
  series={Cambridge Tracts in Mathematics},
  volume={171},
  publisher={Cambridge University Press},
  place={Cambridge},
  date={2007},
  pages={xii+149},
  isbn={978-0-521-87004-7},
  isbn={0-521-87004-6},
  review={\MR {2359514}},
}

\bib{MR2687544}{book}{
  author={Borzellino, Joseph Ernest},
  title={Riemannian geometry of orbifolds},
  note={Thesis (Ph.D.)--University of California, Los Angeles},
  publisher={ProQuest LLC, Ann Arbor, MI},
  date={1992},
  pages={69},
  review={\MR {2687544}},
}

\bib{MR1218706}{article}{
  author={Borzellino, Joseph E.},
  title={Orbifolds of maximal diameter},
  journal={Indiana Univ. Math. J.},
  volume={42},
  date={1993},
  number={1},
  pages={37--53},
  issn={0022-2518},
  review={\MR {1218706 (94d:53053)}},
}

\bib{MR1926425}{article}{
  author={Borzellino, Joseph E.},
  author={Brunsden, Victor},
  title={Orbifold homeomorphism and diffeomorphism groups},
  conference={ title={}, address={Washington, DC}, date={2000}, },
  book={ publisher={World Sci. Publ., River Edge, NJ}, },
  date={2002},
  pages={116--137},
  review={\MR {1926425 (2003g:58013)}},
}

\bib{MR2523149}{article}{
  author={Borzellino, Joseph E.},
  author={Brunsden, Victor},
  title={A manifold structure for the group of orbifold diffeomorphisms of a smooth orbifold},
  journal={J. Lie Theory},
  volume={18},
  date={2008},
  number={4},
  pages={979--1007},
  issn={0949-5932},
  eprint={arXiv:math.DG/0608526},
}

\bib{MR3117354}{article}{
  author={Borzellino, Joseph E.},
  author={Brunsden, Victor},
  title={The stratified structure of spaces of smooth orbifold mappings},
  journal={Commun. Contemp. Math.},
  volume={15},
  date={2013},
  number={5},
  pages={1350018, 37},
  issn={0219-1997},
  review={\MR {3117354}},
  doi={10.1142/S0219199713500181},
  eprint={arXiv:0810.1070},
}

\bib{MR2973378}{article}{
  author={Borzellino, Joseph E.},
  author={Brunsden, Victor},
  title={Elementary orbifold differential topology},
  journal={Topology Appl.},
  volume={159},
  date={2012},
  number={17},
  pages={3583--3589},
  issn={0166-8641},
  review={\MR {2973378}},
  doi={10.1016/j.topol.2012.08.032},
  eprint={arXiv:1205.1156},
}

\bib{MR1950941}{article}{
  author={Chen, Weimin},
  author={Ruan, Yongbin},
  title={Orbifold Gromov-Witten theory},
  conference={ title={Orbifolds in mathematics and physics}, address={Madison, WI}, date={2001}, },
  book={ series={Contemp. Math.}, volume={310}, publisher={Amer. Math. Soc.}, place={Providence, RI}, },
  date={2002},
  pages={25--85},
  review={\MR {1950941 (2004k:53145)}},
}

\bib{MR2104605}{article}{
  author={Chen, Weimin},
  author={Ruan, Yongbin},
  title={A new cohomology theory of orbifold},
  journal={Comm. Math. Phys.},
  volume={248},
  date={2004},
  number={1},
  pages={1--31},
  issn={0010-3616},
  review={\MR {2104605 (2005j:57036)}},
  doi={10.1007/s00220-004-1089-4},
}

\bib{MR3126596}{article}{
  author={Cho, Cheol-Hyun},
  author={Hong, Hansol},
  author={Shin, Hyung-Seok},
  title={On orbifold embeddings},
  journal={J. Korean Math. Soc.},
  volume={50},
  date={2013},
  number={6},
  pages={1369--1400},
  issn={0304-9914},
  review={\MR {3126596}},
}

\bib{MR2962023}{book}{
  author={Choi, Suhyoung},
  title={Geometric structures on 2-orbifolds: exploration of discrete symmetry},
  series={MSJ Memoirs},
  volume={27},
  publisher={Mathematical Society of Japan, Tokyo},
  date={2012},
  pages={xii+171},
  isbn={978-4-931469-68-6},
  review={\MR {2962023}},
}

\bib{MR0448362}{book}{
  author={Hirsch, Morris W.},
  title={Differential topology},
  note={Graduate Texts in Mathematics, No. 33},
  publisher={Springer-Verlag},
  place={New York},
  date={1976},
  pages={x+221},
  review={\MR {0448362 (56 \#6669)}},
}

\bib{MR2879379}{article}{
  author={Kankaanrinta, Marja},
  title={A subanalytic triangulation theorem for real analytic orbifolds},
  journal={Topology Appl.},
  volume={159},
  date={2012},
  number={5},
  pages={1489--1496},
  issn={0166-8641},
  review={\MR {2879379}},
  doi={10.1016/j.topol.2012.01.010},
}

\bib{MR2778793}{article}{
  author={Lerman, Eugene},
  title={Orbifolds as stacks?},
  journal={Enseign. Math. (2)},
  volume={56},
  date={2010},
  number={3-4},
  pages={315--363},
  issn={0013-8584},
  review={\MR {2778793 (2012c:18010)}},
}

\bib{MR1466622}{article}{
  author={Moerdijk, I.},
  author={Pronk, D. A.},
  title={Orbifolds, sheaves and groupoids},
  journal={$K$-Theory},
  volume={12},
  date={1997},
  number={1},
  pages={3--21},
  issn={0920-3036},
  review={\MR {1466622 (98i:22004)}},
  doi={10.1023/A:1007767628271},
}

\bib{Pohl2013}{article}{
  author={Pohl, Anke D.},
  title={Convenient categories of reduced orbifolds},
  status={preprint},
  date={2013},
  eprint={arXiv:1001.0668},
}

\bib{MR0079769}{article}{
  author={Satake, I.},
  title={On a generalization of the notion of manifold},
  journal={Proc. Nat. Acad. Sci. U.S.A.},
  volume={42},
  date={1956},
  pages={359--363},
  issn={0027-8424},
  review={\MR {0079769 (18,144a)}},
}

\bib{MR0095520}{article}{
  author={Satake, Ichir{\^o}},
  title={The Gauss-Bonnet theorem for $V$-manifolds},
  journal={J. Math. Soc. Japan},
  volume={9},
  date={1957},
  pages={464--492},
  issn={0025-5645},
  review={\MR {0095520 (20 \#2022)}},
}

\bib{Thurston78}{manual}{
  author={Thurston, William},
  title={The geometry and topology of 3--manifolds},
  organization={Princeton University Mathematics Department},
  date={1978},
  note={Lecture Notes},
}



\begin{bibdiv}
  \begin{biblist}

    \bibselect{ref}
 
  \end{biblist}
\end{bibdiv}
\end{document}